\documentclass[12pt, reqno]{amsart}

\usepackage{amsmath,amssymb,amsfonts,amsthm}
\usepackage{color}

\usepackage[colorlinks=true,linkcolor=blue,citecolor=blue,urlcolor=blue,pdfborder={0 0 0}]{hyperref}
\usepackage{graphicx}
\usepackage{cleveref}
\usepackage{esvect}

\usepackage{bbm}
\usepackage{enumitem}

\usepackage{caption}
\usepackage{thmtools}
\usepackage{commath}
\usepackage{mathrsfs}

\usepackage{enumitem}
\usepackage{eufrak}

\crefname{theorem}{Theorem}{Theorems}
\crefname{thm}{Theorem}{Theorems}
\crefname{lemma}{Lemma}{Lemmas}
\crefname{lem}{Lemma}{Lemmas}
\crefname{remark}{Remark}{Remarks}
\crefname{prop}{Proposition}{Propositions}
\crefname{defn}{Definition}{Definitions}
\crefname{corollary}{Corollary}{Corollaries}
\crefname{conjecture}{Conjecture}{Conjectures}
\crefname{question}{Question}{Questions}
\crefname{chapter}{Chapter}{Chapters}
\crefname{section}{Section}{Sections}
\crefname{figure}{Figure}{Figures}
\newcommand{\crefrangeconjunction}{--}

\theoremstyle{plain}
\newtheorem{thm}{Theorem}[section]
\newtheorem{lemma}[thm]{Lemma}
\newtheorem{lem}[thm]{Lemma}
\newtheorem{corollary}[thm]{Corollary}
\newtheorem{prop}[thm]{Proposition}

\theoremstyle{definition}

\theoremstyle{remark}
\newtheorem*{remark}{Remark}

\numberwithin{equation}{section}

\renewcommand{\P}{\mathbb P}
\newcommand{\E}{\mathbb E}

\newcommand{\R}{\mathbb R}
\newcommand{\Z}{\mathbb Z}
\newcommand{\N}{\mathbb N}

\newcommand{\D}{\mathbb D}

\newcommand{\eps}{\varepsilon}

\newcommand{\cB}{\mathcal B}
\newcommand{\cC}{\mathcal C}

\newcommand{\cE}{\mathcal E}

\newcommand{\cI}{\mathcal I}

\newcommand{\cM}{\mathcal M}

\newcommand{\cS}{\mathcal S}

\newcommand{\sA}{\mathscr A}
\newcommand{\sB}{\mathscr B}
\newcommand{\sC}{\mathscr C}

\newcommand{\bbG}{\mathbb G}

\usepackage{stmaryrd}
\usepackage[margin=1.3in]{geometry}

\usepackage{sidecap}
\title{Boundaries of Planar Graphs: A Unified Approach}

\newcommand{\Ceff}{\sC_{\mathrm{eff}}}

\newcommand{\CeffF}{\sC_{\mathrm{eff}}^F}

\author{Tom Hutchcroft and Yuval Peres}

\begin{document}

\maketitle
\begin{abstract} 
We give a new proof that the Poisson boundary of a planar graph coincides with the boundary of its square tiling and with the boundary of its circle packing, originally proven by Georgakopoulos \cite{G13} and Angel, Barlow, Gurel-Gurevich and Nachmias~\cite{ABGN14} respectively. Our proof is robust, and also allows us to identify the Poisson boundaries of graphs that are rough-isometric to planar~graphs.  

We also prove that the boundary of the square tiling of a bounded degree plane triangulation coincides with its Martin boundary. This is done by comparing the square tiling of the triangulation with its circle packing.
\end{abstract}

\section{Introduction}

Square tilings of planar graphs were introduced by Brooks, Smith, Stone and Tutte~\cite{BSST40}, and are closely connected to random walk and potential theory on planar graphs.  Benjamini and Schramm \cite{BS96b} extended the square tiling theorem to infinite, \emph{uniquely absorbing} plane graphs (see \cref{Background:Embedding}). These square tilings take place on the cylinder $\R/\eta\Z \times [0,1]$, where $\eta$ is the effective conductance to infinity from some fixed root vertex $\rho$ of $G$. They also proved that the random walk on a transient, bounded degree, uniquely absorbing plane graph converges to a point in the boundary of the cylinder $\R/\eta\Z\times\{1\}$, and that the limit point of a random walk started at $\rho$ is distributed according to the Lebesgue measure on the boundary of the cylinder. 

 Benjamini and Schramm \cite{BS96b} applied their convergence result to deduce that every transient, bounded degree planar graph admits \emph{non-constant bounded harmonic functions}. 
Recall that a function $h:V\to\R$ on the state space of a Markov chain $(V,P)$ is \textbf{harmonic} if 
\[h(u)=\sum_{v\sim u}P(u,v)h(v)\]
for every vertex $u\in V$, or equivalently if $\langle h(X_n)\rangle_{n\geq0}$ is a martingale when $\langle X_n\rangle_{n\geq0}$ is a trajectory  of the Markov chain. 
If $G$ is a transient, uniquely absorbing, bounded degree plane graph, then for each bounded Borel function $f:\R/\eta\Z\to\R$, we define a harmonic function $h$ on $G$ by setting
\[h(v) = E_v\left[f\left(\lim_{n\to\infty} \theta(X_n)\right)\right]\] 
for each $v\in V$, where $E_v$ denotes the expectation with respect to a random walk $\langle X_n \rangle_{n\geq0}$ started at $v$ and $\theta(v)$ is the horizontal coordinate associated to the vertex $v$ by the square tiling of $G$ (see \cref{Background:Tiling}).
Georgakopoulos \cite{G13} proved that moreover \emph{every} bounded harmonic function on $G$ may be represented this way, answering a question of Benjamini and Schramm \cite{BS96b}. 
In other words, Georgakopoulos's theorem identifies the geometric boundary of the square tiling of $G$ with the \emph{Poisson boundary} of $G$ (see \cref{sec:Poisson}). 
  Probabilistically, this means that the tail $\sigma$-algebra of the random walk $\langle X_n \rangle_{n\geq0}$ is trivial conditional on the limit of $\theta(X_n)$.

In this paper, we give a new proof of Georgakopoulos's theorem. We state our result in the natural generality of plane networks. Recall that a \textbf{network} $(G,c)$ is a connected, locally finite graph $G=(V,E)$, possibly containing self-loops and multiple edges, together with a function $c:E\to(0,\infty)$ assigning a positive \textbf{conductance} to each edge of $G$. The conductance $c(v)$ of a vertex $v$ is defined to be the sum of the conductances of the edges emanating from $v$, and for each pair of vertices $u,v$ the conductance $c(u,v)$ is defined to be the sum of the conductances of the edges connecting $u$ to $v$.  The random walk on the network is the Markov chain with transition probabilities $p(u,v)=c(u,v)/c(u)$. 
Graphs without specified conductances are considered networks by setting $c(e)\equiv1$. We will usually suppress the notation of conductances, and write simply $G$ for a network. Instead of square tilings, general plane networks are associated to \emph{rectangle tilings}, see \cref{Background:Tiling}. See \cref{Background:Embedding} for detailed definitions of plane graphs and networks. For each vertex $v$ of $G$, $I(v)\subseteq \R/\eta\Z$ is an interval associated to $v$ by the rectangle tiling of $G$.

\begin{thm}[Identification of the Poisson boundary]\label{Thm:generalPoisson}
Let $G$ be a plane network and let $\cS_\rho$ be the rectangle tiling of $G$ in the cylinder $\R/\eta\Z\times[0,1]$. Suppose that $\theta(X_n)$ converges to a point in $\R/\eta\Z$  and that $\mathrm{length}(I(X_n))$ converges to zero almost surely as $n$ tends to infinity. 
Then for every bounded harmonic function $h$ on $G$, there exists a bounded Borel function $f:\R/\eta\Z \to \R$ such that 
\[ h(v) =  E_v\left[f\left(\lim_{n\to\infty} \theta(X_n)\right)\right]. \]
for every $v\in V$. That is, the geometric boundary of the rectangle tiling of $G$ coincides with the Poisson boundary of $G$.
\end{thm}

The convergence theorem of Benjamini and Schramm \cite{BS96b} implies that the hypotheses of \cref{Thm:generalPoisson} are satisfied when $G$ has bounded degrees.

\subsection{Circle packing} An alternative framework in which to study harmonic functions on planar graphs is given by the circle packing theorem. A \textbf{circle packing} is a collection $\cC$ of non-overlapping (but possibly tangent) discs in the plane. Given a circle packing $\cC$, the \textbf{tangency graph} of $\cC$ is defined to be the graph with vertices corresponding to the discs of $\cC$ and with two vertices adjacent if and only if their corresponding discs are tangent. The tangency graph of a circle packing is clearly planar, and can be drawn with straight lines between the centres of tangent discs in the packing. The Koebe-Andreev-Thurston Circle Packing Theorem \cite{K36,Th78} states conversely that every finite, simple (i.e., containing no self-loops or multiple edges),  planar graph may be represented as the tangency graph of a circle packing. If the graph is a triangulation (i.e., every face has three sides), its circle packing is unique up to M\"obius transformations and reflections.  We refer the reader to \cite{St05} and \cite{Rohde11} for background on circle packing.

The \textbf{carrier} of a circle packing is defined to be the union of all the discs in
the packing together with the bounded regions that are disjoint from the discs in the packing and are enclosed by
the some set of discs in the packing corresponding to a face of the tangency graph.
Given some planar domain $D$, we say that a circle packing is \textbf{in
  $D$} if its carrier is $D$.

The circle packing theorem was extended to infinite planar graphs by He and Schramm \cite{HeSc}, who proved that every proper plane triangulation admits a locally finite circle packing in the plane or the disc, but not both. We call a triangulation of the plane \textbf{CP parabolic} if it can be circle packed in the plane and \textbf{CP hyperbolic} otherwise.  He and Schramm also proved that a bounded degree simple triangulation of the plane is CP parabolic if and only if it is recurrent for the simple random walk.  

\begin{figure}[t]
\centering
\includegraphics[trim = 2em .5em 2em .5em, clip, height=.45\textwidth]{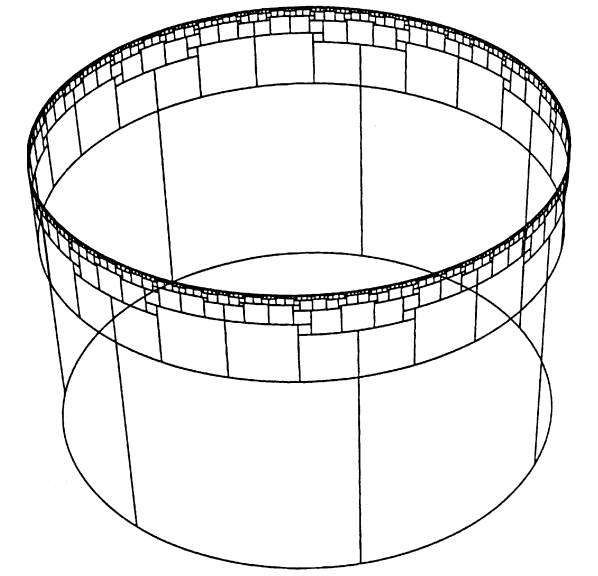} \qquad \includegraphics[trim = 1em 10.5em 3em 10em, clip,height=.45\textwidth]{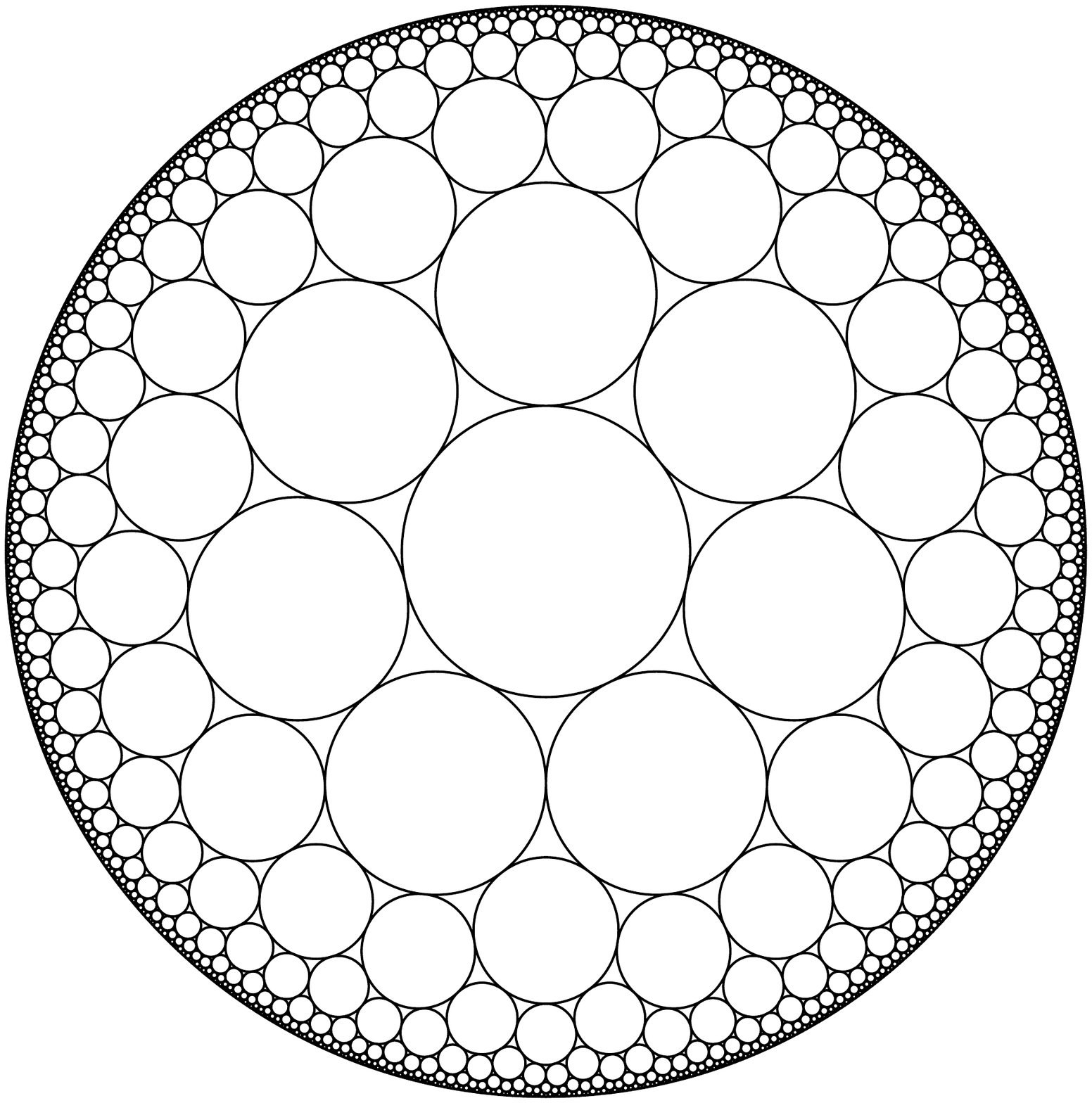}
\caption{The square tiling and the circle packing of the 7-regular hyperbolic triangulation.} 
\end{figure}

Benjamini and Schramm \cite{BS96a} proved that, when a bounded degree, CP hyperbolic triangulation is circle packed in the disc, the simple random walk converges to a point in the boundary of the disc and the law of the limit point is non-atomic and has full support. 
 Angel, Barlow, Gurel-Gurevich and Nachmias \cite{ABGN14} later proved that, under the same assumptions, the boundary of the disc is a realisation of the Poisson boundary of the triangulation. 
  These results were extended to unimodular random rooted triangulations of unbounded degree by Angel, Hutchcroft, Nachmias and Ray \cite{AHNR15}. Our proof of \cref{Thm:generalPoisson} is adapted from the proof of \cite{AHNR15}, and also yields a new proof of the Poisson boundary result of \cite{ABGN14}, which follows as a special case of both \cref{thm:roughembeddings,thm:markovpoisson} below.



\subsection{Robustness under rough isometries}
The proof of \cref{Thm:generalPoisson} is quite robust, and also allows us to characterise the Poisson boundaries of certain non-planar networks. 
Let $G=(V,E)$ and $G'=(V',E')$ be two graphs, and let $d$ and $d'$ denote their respective graph metrics. Recall that a function $\phi:V\to V'$ is a \textbf{rough isometry} if there exist positive constants $\alpha$ and $\beta$ such that the following conditions are satisfied:
\begin{enumerate}
\item (\textbf{Rough preservation of distances.}) For every pair of vertices $u$ and $v$ in $G$,
 \[\alpha^{-1}d(u,v) -\beta \leq d'(\phi(u),\phi(v)) \leq \alpha d(u,v)+\beta.\]
 \item (\textbf{Almost surjectivity.}) For every vertex $v' \in V'$, there exists a vertex $v\in V$ such that $d(\phi(v),v')\leq \beta$.
\end{enumerate}
Rough isometries were introduced  by Kanai \cite{Kanai85} and Gromov \cite{Gromov81,Gromov87}. For background on rough isometries and their applications, see \cite[\S 2.6]{LP:book} and \cite[\S7]{Soardibook}. We say that a network $G=(V,E)$ has \textbf{bounded local geometry} if there exists a constant $M$ such that $\deg(v)\leq M$ for all $v\in V$ and $M^{-1}\leq c(e)\leq M$ for all $e \in E$.

 Benjamini and Schramm \cite{BS96a} proved that every transient network of bounded local geometry that is rough isometric to a planar graph admits non-constant bounded harmonic functions. 
In general, however, rough isometries do not preserve the property of admitting non-constant bounded harmonic functions \cite[Theorem 3.5]{BS96a}, and consequently do not preserve Poisson boundaries. 

Our next theorem establishes that, for a bounded degree graph $G$ roughly isometric to a bounded degree proper plane graph $G'$, the Poisson boundary of $G$ coincides with the geometric boundary of a suitably chosen embedding of $G'$, so that the same embedding gives rise to a realisation of the Poisson boundaries of both $G$ and $G'$. See \cref{Background:Embedding} for the definition of an embedding of a planar graph.

\begin{thm}[Poisson boundaries of roughly planar networks]\label{thm:roughembeddings} Let $G$ be a transient network with bounded local geometry such that there exists a proper plane graph $G'$ with bounded degrees and a rough isometry $\phi: G \to G'$. Let $\langle X_n \rangle_{n\geq0}$ be a random walk on $G$. Then there exists an embedding $z$ of $G'$ into $\D$ such that $z\circ\phi(X_n)$ converges to a point in $\partial \D$ and the law of the limit point is non-atomic. Moreover, for every such embedding $z$ and for every bounded harmonic function $h$ on $G$, there exists  a bounded Borel function $f:\partial\D \to \R$ such that 
\[ h(v) =  E_v\left[f\left(\lim_{n\to\infty} z\circ\phi(X_n)\right)\right]. \]
for every $v\in V$. That is, the geometric boundary of the disc coincides with the Poisson boundary of $G$.
\end{thm}
\noindent
The part of \cref{thm:roughembeddings} concerning the existence of an embedding is implicit in~\cite{BS96a}.
\medskip

A further generalisation of \cref{Thm:generalPoisson} concerns embeddings of possibly irreversible planar Markov chains: The only changes required to the proof of \cref{Thm:generalPoisson} in order to prove the following are notational.
\begin{thm}\label{thm:markovpoisson}
Let $(V,P)$ be a Markov chain such that the graph 
\[G=\left(V,\{(u,v) \in V^2 : P(u,v)>0 \;\mathrm{ or }\; P(v,u)>0\}\right)\]
is planar.  Suppose further that there exists a vertex $\rho \in V$ such that for every $v\in V$ there exists $n$ such that $P^n(\rho,v)>0$, and let $\langle X_n \rangle_{n\geq0}$ be a trajectory of the Markov chain. Let $z$ be a (not necessarily proper) embedding of $G$ into the unit disc $\D$ such that $\langle z(X_n) \rangle_{n\geq0}$ converges to a point in $\partial \D$ almost surely and the law of the limit point is non-atomic. 
Then for every bounded harmonic function $h$ on $(V,P)$, there exists a bounded Borel function $f:\partial\D \to \R$ such that 
\[ h(v) =  E_v\left[f\left(\lim_{n\to\infty} z(X_n)\right)\right]. \]
for every $v\in V$.
\end{thm}

\subsection{The Martin boundary}

  In \cite{ABGN14} it was also proven that the \emph{Martin boundary} of a bounded degree CP hyperbolic triangulation can be identified with the geometric boundary of its circle packing. 
 Recall that a function $g:V\to\R$ on a network $G$ is \textbf{superharmonic} if 
\[g(u) \geq \frac{1}{c(u)}\sum_{v\sim u}c(u,v)g(v)\]
for every vertex $u\in V$.  
Let $\rho$ be a fixed vertex of $G$ and consider the space $\mathfrak{S}^+$ of positive superharmonic functions $g$ on $G$ such that $g(\rho)=1$, which is a convex, compact subset of the space of functions $V\to\R$ equipped with the product topology (i.e.~ the topology of pointwise convergence). For each $v\in V$, let $P_v$ be the law of the random walk on $G$ started at $v$. We can embed $V$ into $\mathfrak{S}^+$ by sending each vertex $u$ of $G$ to its \textbf{Martin kernel}
\[M_u(v):=\frac{E_v\left[\#(\text{visits to }u)\right]}{E_\rho\left[\#(\text{visits to }u)\right]}=\frac{P_v(\text{hit }u)}{P_\rho(\text{hit }u)}.\]
The \textbf{Martin compactification} $\cM (G)$ of the network $G$ is defined as the closure of $\{M_u : u \in V\}$ in $\mathfrak{S}^+$, and the \textbf{Martin boundary} $\partial \cM (G)$ of the network $G$ is defined to be the complement of the image of $V$,  \[\partial\cM(G):=\cM(G) \setminus\{M_u : u \in V\}.\]
See \cite{Dynkin69,RodgersWilliamsbookV1,Woess} for background on the Martin boundary. 

 
 Our next result is that, for a triangulation of the plane with bounded local geometry, the geometric boundary of the square tiling coincides with the Martin boundary. 

\begin{thm}[Identification of the Martin boundary]\label{Thm:Martin}
Let $T$ be a transient, simple,  proper plane triangulation with bounded local geometry. Let $\cS_\rho$ be a square tiling of $T$ in a cylinder $\R/\eta\Z\times[0,1]$.
Then 
\begin{enumerate}
\item A sequence of vertices $\langle v_n \rangle_{n\geq0}$ in $T$ converges to a point in the Martin boundary of $T$ if and only if $y(v_n)\to1$ and $\theta(v_n)$ converges to a point in $\R/\eta\Z$. 
\item The map
\[ M:\theta \longmapsto M_\theta :=\lim_{n\to\infty} M_{v_n} \text{ where } \left(\theta(v_n),y(v_n)\right)\to(\theta,1),\]
which is well-defined by $(1)$, is a homeomorphism from $\R/\eta\Z$ to the Martin boundary $\partial \cM (T)$ of $T$.
\end{enumerate}
That is, the geometric boundary of the rectangle tiling of $G$ coincides with the Martin boundary of $G$. 
\end{thm}

This will be deduced from the analogous statement for circle packings \cite[Theorem 1.2]{ABGN14} together with the following theorem, which states that for a bounded degree triangulation $T$, the square tiling and circle packing of $T$ define equivalent compactifications of~$T$. 

\begin{thm}[Comparison of square tiling and circle packing]\label{thm:comparison} Let $T$ be a transient, simple, proper plane triangulation with bounded local geometry. Let $\cS_\rho$ be a rectangle tiling of $T$ in the cylinder $\R/\eta\Z\times[0,1]$ and let $\cC$ be a circle packing of $T$ in $\D$ with associated embedding $z$. 
Then 
\begin{enumerate}\item 
A sequence of vertices $\langle v_n \rangle_{n\geq0}$ in $T$ converges to a point in $\partial \D$ if and only if $y(v_n)\to1$ and $\theta(v_n)$ converges to a point in $\R/\eta\Z$.
\item The map 
\[ \xi \mapsto \lim_{i\to\infty}\theta(v_i) \text{ where } z(v_i)\to\xi, \]
which is well-defined by $(1)$, 
is a homeomorphism from $\partial\D$ to $\R/\eta\Z$. 
\end{enumerate}
\end{thm}

\cref{thm:comparison} also allows us to deduce the Poisson boundary results of \cite{G13} and \cite{ABGN14} from each other in the case of bounded degree triangulations. 

\cref{Thm:Martin} has the following immediate corollaries by standard properties of the Martin compactification.

\begin{corollary}[Continuity of harmonic densities]\label{Cor:densities} Let $T$ be a transient proper simple plane triangulation with bounded local geometry, and let $\cS_\rho$ be a rectangle tiling of $T$ in the cylinder $\R/\eta\Z\times[0,1]$. For each vertex $v$ of $T$, let $\omega_v$ denote the harmonic measure from $v$, defined by
\[\omega_v(\sA) := P_v\Big(\lim_{n\to\infty}\theta(X_n)\in\sA\Big)\]
for each Borel set $\sA\subseteq\R/\eta\Z$, and let $\lambda=\omega_\rho$ denote the Lebesgue measure on $\R/\eta\Z$.
Then for every $v$ of $T$, the density of the harmonic measure from $v$ with respect to Lebesgue is given by
\[ \frac{d\omega_v}{d\lambda}\!\left(\theta\right) = M_\theta(v) \]
which is continuous with respect to $\theta$.
\end{corollary}

\begin{corollary}[Representation of positive harmonic functions]\label{Cor:martinrep}
Let $T$ be a transient proper simple plane triangulation with bounded local geometry, and let $\cS_\rho$ be a rectangle tiling of $T$ in the cylinder $\R/\eta\Z\times[0,1]$. Then for every positive harmonic function $h$ on $T$, there exists a unique measure $\mu$ on $\R/\eta\Z$ such that
\[ h(v) = \int_{\R/\eta\Z}\! M_\theta(v) \dif\mu(\theta).\]
for every $v\in V$.
\end{corollary}



\section{Background}

\subsection{Notation} We use $e$ to denote both oriented and unoriented edges of a graph or network. An oriented edge $e$ is oriented from its tail $e^-$ to its head $e^+$. Given a network $G$ and a vertex $v$ of $G$, we write $P_v$ for the law of the random walk on $G$ started at $v$ and $E_v$ for the associated expectation operator. 

\subsection{Embeddings of Planar Graphs}\label{Background:Embedding}
Let $G=(V,E)$ be a graph.
  For each edge $e$, choose an orientation of $e$ arbitrarily and let $I(e)$ be an isometric copy of the interval $[0,1]$. 
The metric space $\mathbb{G}=\bbG(G)$ is defined to be the quotient of the union $\bigcup_e I(e) \cup V$, where we identify the endpoints of $I(e)$ with the vertices $e^-$ and $e^+$ respectively, and is equipped with the path metric. 
 An \textbf{embedding of $G$} into a surface $S$ is a continuous, injective map $z:\bbG\to S$. The embedding is \textbf{proper} if every compact subset of $S$ intersects at most finitely many edges and vertices of $z(\bbG)$. 
A graph is planar if and only if it admits an embedding into $\R^2$.  A \textbf{plane graph} is a planar graph together with an embedding $\bbG(G)\to \R^2$; it is a \textbf{proper plane graph} if the embedding $z$ is proper. 
A \textbf{(proper) plane network} is a (proper) plane graph together with an assignment of conductances $c:E\to(0,\infty)$. 
A \textbf{proper plane triangulation} is a proper plane network in which every face (i.e connected component of $\R^2 \setminus z(\bbG)$) has three sides. 

A set of vertices $W \subseteq V$ is said to be \textbf{absorbing} if with positive probability the random walk 
$\langle X_n \rangle_{n\geq0}$ on $G$ is contained in $W$ for all $n$ greater than some random $N$.  
A plane graph $G$ is said to be \textbf{uniquely absorbing} if for every finite subgraph $G_0$ of $G$, there is exactly one connected component $D$ of  $\R^2\setminus \bigcup\{z(I(e)) : e \in G_0\}$ such that the set of vertices $\{v\in V: z(v) \in D\}$ is absorbing.


\subsection{Square Tiling}\label{Background:Tiling}


Let $G$ be a transient, uniquely absorbing plane network and let $\rho$ be a vertex of $G$. For each $v\in V$ let $y(v)$ denote the probability that the random walk on $G$ started at $v$ never visits $\rho$, and let 
\[ \eta := \sum_{u\sim \rho} c(\rho,u)y(u).\]
Let $\R/\eta\Z$ be the circle of length $\eta$. Then there exists a set \[\cS_\rho=\{ S(e) : e \in E \}\] 
 such that:
\begin{enumerate}
\item For each oriented edge $e$ of $G$ such that $y(e^+)\geq y(e^-)$, $S(e)\subseteq \R/\eta\Z\times[0,1)$ is a rectangle of the form \[ S(e)=I(e) \times \left[y(e^-),y(e^+)\right]\] where $I(e)\subseteq \R/\eta\Z$ is an interval of length \[\mathrm{length}\left(I(e)\right):=c(e)\left(y(e^+)-y(e^-)\right).\] If $e$ is such that $y(e^+) < y(e^-)$, we define $I(e)=I(-e)$ and $S(e)=S(-e)$. In particular, the aspect ratio of $S(e)$ is equal to the conductance $c(e)$ for every edge $e\in E$.
\item The interiors of the rectangles $S(e)$ are disjoint, and the union $\bigcup_e S(e)=\R/\eta\Z\times[0,1)$.
\item For every vertex $v\in V$, the set 
$I(v) = \bigcup_{e^- =v}I(e)$  is an interval and is equal to  $\bigcup_{e^+=v} I(e)$.  
\item For almost every $\theta \in C$ and for every $t\in[0,1)$, the line segment $\{\theta\}\times [0,t]$ intersects only finitely many rectangles of $\cS_\rho$.
\end{enumerate}
Note that the rectangle corresponding to an edge through which no current flows is degenerate, consisting of a single point. 
The existence of the above tiling  was proven by Benjamini and Schramm \cite{BS96b}. Their proof was stated for the case $c\equiv 1$ but extends immediately to our setting, see \cite{G13}.

Let us also note the following property of the rectangle tiling, which follows from the construction given in~\cite{BS96b}.
\begin{enumerate}
\item[(5)]  For each two edges $e_1$ and $e_2$ of $G$, the interiors of the vertical sides of the rectangles $S(e_1)$ and $S(e_2)$ have a non-trivial intersection only if $e_1$ and $e_2$ both lie in the boundary of some common face $f$ of $G$. 
\end{enumerate}

For each $v\in V$, we let $\theta(v)$ be a point chosen arbitrarily from $I(v)$.
\medskip
Let $G$ be a uniquely absorbing proper plane network. 
Benjamini and Schramm~\cite{BS96b} proved that if $G$ has bounded local geometry and $\langle X_n \rangle_{n\geq0}$ is a random walk on $G$ started at $\rho$, then  $\theta(X_n)$ converges to a point in $\R/\eta\Z$ and the law of the limit point is Lebesgue (their proof is given for bounded degree plane graphs but extends immediately to this setting). An observation of Georgakopoulos \cite[Lemma 6.2]{G13} implies that, more generally, whenever $G$ is such that $\theta(X_n)$ converges to a point in $\R/\eta\Z$ and $\mathrm{length}(I(X_n))$ converges to zero almost surely, the law of the limit point is Lebesgue.

\section{The Poisson Boundary}
\label{sec:Poisson}

Let $(V,P)$ be a Markov chain. 
Harmonic functions on $(V,P)$ encode asymptotic behaviours of a trajectory~$\langle X_n \rangle_{n\geq0}$ as follows. Let $\Omega$ denote the path space
\[ \Omega = \left\{ \langle x_i \rangle_{i\geq0} \in V^\N : p(x_i,x_{i+1})>0 \quad \forall i\geq0\right\}\]
and let $\cB$ denote the Borel $\sigma$-algebra for the product topology on $\Omega$. Let $\cI$ denote the invariant $\sigma$-algebra 
\[ \cI = \left\{ \sA \in \cB : \langle x_i \rangle_{i\geq0} \in \sA \iff \langle x_{i+1} \rangle_{i\geq0} \in \sA \quad \forall \langle x_i \rangle_{i\geq 0} \in \Omega\right\}. \]
Assume that there exists a vertex $\rho\in V$ from which all other vertices are reachable: 
\[\forall v \in V \,\, \exists k\geq 0 \text{ such that } p_k(\rho,v)>0\]
which is always satisfied when $P$ is the transition operator of the random walk on a network~$G$.
Then there exists an invertible linear transformation $\mathbf{H}$ between $L^\infty(\Omega,\cI,P_\rho)$ and the space of bounded harmonic functions on $G$ defined as follows \cite[Proposition 14.12]{blackwell1955transient,LP:book}:
\begin{align} \mathbf{H}: f &\longmapsto \mathbf{H}f(v) = E_v\left[f\left(\langle X_n \rangle_{n\geq0}\right)\right]\nonumber\\   \mathbf{H}^{-1}: h &\longmapsto
 \tilde h\left(\langle x_i \rangle_{i\geq0}\right) = \lim_{i\to\infty}h(x_i) \label{eq:H}\end{align}
where the above limit exists for $P_\rho$-a.e.\ sequence $\langle x_i \rangle_{i\geq0}$ by the martingale convergence theorem.

\begin{prop}[Path-hitting criterion for the Poisson boundary]\label{thm:criterion}
Let $(V,P)$ be a Markov chain and let $\langle X_n \rangle_{n\geq0}$ be a trajectory of the Markov chain. Suppose that $\psi:V\to \mathbb{M}$ is a function from $V$ to a metric space $\mathbb{M}$ such that $\psi(X_n)$ converges to a point in $\mathbb{M}$ almost surely. For each $k\geq 0$, let $\langle Z^k_{j} \rangle_{j\geq 0}$ be a trajectory of the Markov chain started at $X_k$ that is conditionally independent of $\langle X_n \rangle_{n\geq0}$ given $X_k$, and let $\P$ denote the joint distribution of $\langle X_n \rangle_{n\geq0}$ and each of the $\langle Z^k_m \rangle_{m\geq0}$.  
Suppose that almost surely, for every path $\langle v_i \rangle_{i\geq0}\in\Omega$ started at $\rho$ such that $\lim_{i\to\infty}\psi(v_i)=\lim_{n\to\infty}\psi(X_n)$, we have that 
\begin{equation}\label{eq:pathassumption}\limsup_{k \to \infty}\P\left(\langle Z^k_m \rangle_{m\geq0} \text{ hits } \{v_i : i\geq 0\} \,\middle|\, \langle X_n \rangle_{n\geq0} \right)>0.\end{equation}
Then for every bounded harmonic function $h$ on $G$, there exists a bounded Borel function $f:\mathbb{M}\to\R$ such that 
\[ h(v) =  E_v\left[f\left(\lim_{n\to\infty} \psi(X_n)\right)\right] \]
for every $v\in V$.
\end{prop}

Note that it suffices to define $f$ on the support of the law of $\lim_{n\to\infty}\psi(X_n)$, which is contained in the set of accumulation points of $\{\psi(v) : v\in V\}$.

 A consequence of the correspondence \eqref{eq:H} is that, to prove \cref{thm:criterion}, it suffices to prove that  for every invariant event $\sA\in\cI$, there exists a Borel set $\sB \subseteq \mathbb{M}$ such that
\[P_v\left(\sA\triangle\big\{\lim_{n\to\infty}\psi(X_n) \in \sB\big\}\right)=0 \text{ for every vertex $v\in V$.}\]

\begin{proof}[Proof of \cref{thm:criterion}]
Let $\sA\in\cI$ be an invariant event and  let $h$ be the harmonic function
\[ h(v):=P_v(\langle X_n \rangle_{n\geq0}\in\sA).\] 
L\'evy's 0-1 law implies that 
\[h(X_n) \xrightarrow[n\to\infty]{a.s.}\mathbbm{1}\left(\langle X_n \rangle_{n\geq0} \in \sA\right)\]
and so it suffices to exhibit a Borel set $\sB\subseteq \mathbb{M}$ such that
 \begin{multline*}P_\rho\left(\sA\triangle\big\{\lim_{n\to\infty}\psi(X_n)\in\sB\big\}\right)\\=P_\rho\left(\big\{\limsup_{n\to\infty} h(X_n)>0\big\}\triangle\big\{\lim_{n\to\infty}\psi(X_n)\in\sB\big\}\right)=0.\end{multline*}
We may assume that $P_\rho(\sA)>0$, otherwise the claim is trivial.

\medskip

Let $d_\mathbb{M}$ denote the metric of $\mathbb{M}$. 
For each natural number $m>0$, let $N(m)$ be the smallest natural number such that \[P_\rho\left(\exists n \geq N(m) \text{ such that }d_{\mathbb{M}}\left(\psi(X_n),\lim_{k\to\infty} \psi(X_k)\right)\geq\frac{1}{m} \right)\leq 2^{-m}.\]
For each $n$, let $m(n)$ be the largest $m$ such that $n\geq N(m)$, so that $m(n)\to\infty$ as $n\to\infty$. Borel-Cantelli implies that, for all but finitely many $m$, there does not exist an $n \geq N(m)$ such that 
\[  d_{\mathbb{M}}\left(\psi(X_n),\lim_{k\to\infty} \psi(X_k)\right) > \frac{1}{m}, 
\]
and it follows that
\[ d_{\mathbb{M}}\left(\psi(X_n),\lim_{k\to\infty} \psi(X_k)\right) \leq \frac{1}{m(n)} \]
for all but finitely many $n$ almost surely.

 Define the set $\sB\subseteq \mathbb{M}$ by
\[ \hspace{-.3em}\sB :=\left\{x \in \mathbb{M} : \begin{array}{l} \exists \text{ a path }\langle v_i \rangle_{i\geq0} \text{ in $G$ with $v_0=\rho$ such} \text{ that } d_{\mathbb{M}}(\psi(v_i),x)\leq 1/m(i)\\ \text{for all but} \text{ finitely } \text{many } i \text{ and } \inf_{i\geq 0}h(v_i)>0\end{array}\right\}.\]
To see that $\sB$ is Borel, observe that it may be written in terms of closed subsets of $\mathbb{M}$ as follows 
\[\sB = \bigcup_{k\geq0}\bigcup_{j\geq0}\bigcap_{I\geq j} \left\{x\in \mathbb{M} : \begin{array}{l} \exists \text{ a path }\langle v_i \rangle_{i=0}^I \text{ in $G$ with $v_0=\rho$} \text{ such that}\\ d_{\mathbb{M}}(\psi(v_i),x)\leq 1/m(i) \text{ for all } j \leq i \leq I \text{ and}\\ h(v_i)\geq 1/k \text{ for all } 0 \leq i \leq I \end{array}\right\}.\]
It is immediate that $\lim_{n\to\infty}\psi(X_n)\in\sB$ almost surely on the event that $h(X_n)$ converges to 1: simply take  $\langle v_i \rangle_{i\geq0}=\langle X_i \rangle_{i\geq 0}$ as the required path. In particular, the event $\{\lim_{n\to\infty}\psi(X_n)\in\sB\}$ has positive probability.  

We now prove conversely that $\liminf_{n\to\infty} h(X_n)>0$ almost surely on the event that $\lim_{n\to\infty}\psi(X_n)\in\sB$.
Condition on this event, so that there exists a path $\langle v_i \rangle_{i\geq0}$ in $G$ starting at $\rho$ such that $\lim_{i\to\infty}\psi(v_i) = \lim_{n\to\infty}\psi(X_n)$ and $\inf_{i\geq0} h(v_i) >0$. Fix one such path. 
Applying the optional stopping theorem to $\langle h(Z^k_m) \rangle_{m\geq0}$, we have
\begin{align*} h(X_k) &\geq \P\left(\langle Z^k_m \rangle_{m\geq0} \text{ hits } \{v_i : i\geq 0\} \,\middle|\, \langle X_n \rangle_{n\geq0} \right) \cdot \inf\{h(v_i) : i\geq 0\} \end{align*}
and so, by our assumption \eqref{eq:pathassumption}, we have that
\[ \limsup_{k\to\infty} h(X_k) \geq \limsup_{k\to\infty}\P\left(\langle Z^k_m \rangle_{m\geq0} \text{ hits } \{v_i : i\geq 0\} \,\middle|\, \langle X_n \rangle_{n\geq0} \right) \cdot \inf\{h(v_i) : i\geq 0\}\]
is positive almost surely.
\end{proof}

We remark that controlling the rate of convergence of the path in the definition of $\sB$ can be avoided by invoking the theory of universally measurable sets.

We now apply the criterion given by \cref{thm:criterion} to prove \cref{Thm:generalPoisson} and \cref{thm:roughembeddings}.

\begin{proof}[Proof of \cref{Thm:generalPoisson}]
Let $\langle X_n \rangle_{n\geq0}$ and $\langle Y_n \rangle_{n\geq0}$ be independent random walks on $G$ started at $\rho$, and for each $k\geq 0$ let $\langle Z^k_{j} \rangle_{j\geq 0}$ be a random walk on $G$ started at $X_k$ that is conditionally independent of $\langle X_n \rangle_{n\geq0}$ and $\langle Y \rangle_{n\geq0}$ given $X_k$.   Let $\P$ denote the joint distribution of $\langle X_n \rangle_{n\geq0}$, $\langle Y_n \rangle_{n\geq0}$ and all of the random walks $\langle Z^k_m \rangle_{m\geq0}$.  
Given two points $\theta_1,\theta_2 \in \R/\eta\Z$, we denote by $(\theta_1,\theta_2)\subset \R/\eta\Z$ the open arc between $\theta_1$ and $\theta_2$  in the counter-clockwise direction. For each such interval $(\theta_1,\theta_2)\in \R/\eta\Z$, let
\[q_{(\theta_1,\theta_2)}(v):=P_v\left(\lim_{n\to\infty}\theta(X_n)\in(\theta_1,\theta_2)\right).\]
be the probability that a random walk started at $v$ converges to a point in the interval~$(\theta_1,\theta_2)$.

Since the law of $\lim_{n\to\infty}\theta(X_n)$ is Lebesgue and hence non-atomic, the two random variables $\theta^+:=\lim_{n\to\infty}\theta(X_n)$ and $\theta^-:=\lim_{n\to\infty}\theta(Y_n)$ are almost surely distinct. We can therefore write $\R/\eta\Z\setminus\{\theta^+,\theta^-\}$ as the union of the two disjoint non-empty intervals $\R/\eta\Z\times\{1\} \setminus \{\theta^+,\theta^-\}=(\theta^+,\theta^-)\cup (\theta^-,\theta^+)$. Let 
\[Q_k:=q_{(\theta^-,\theta^+)}(X_k)= \P\left(\lim_{m\to\infty} \theta(Z^k_m) \in (\theta^-,\theta^+) \,\middle|\, \langle X_n \rangle_{n\geq0}, \langle Y_n \rangle_{n\geq0} \right)\] be the probability that a random walk started at $X_k$, that is conditionally independent of $\langle X_n \rangle_{n\geq0}$ and $\langle Y_n \rangle_{n\geq0}$ given $X_k$, converges to a point in the interval~$(\theta^-,\theta^+)$. 

We claim that the random variable $Q_k$ is uniformly distributed on $[0,1]$ conditional on $\langle X_n \rangle_{n=0}^k$ and $\langle Y_n \rangle_{n\geq0}$. Indeed, since the law of $\theta^+$ given $X_k$ is non-atomic, for each $s\in[0,1]$ there exists\footnote{In the present setting, $\theta^s$ is unique since the ;aw of $\theta^+$ given $X_k$ has full support.} $\theta^s=\theta^s(X_k,\theta^-)\in\R/\eta\Z$ such that 
\[\P\left(\theta^+\in(\theta^-,\theta^s)\,\middle|\, X_k,\theta^-\right)=s.\]
 The claim follows by observing that
\[\P\left(Q_k \in [0,s] \,\middle|\, \langle X_n \rangle_{n=0}^k,\langle Y_n \rangle_{n\geq0}\right) = \P\left(\theta^+\in(\theta^-,\theta^s)\mid X_k,\theta^-\right) = s. \]
As a consequence, Fatou's lemma implies that for every $\eps>0$,
\begin{align*}\P(Q_k \in [\eps,1-\eps] \text{ infinitely often}) &= \E\left[\limsup_{k\to\infty}\mathbbm{1}\left(Q_k\in[\eps,1-\eps]\right)\right]\\
&\geq \limsup_{k\to\infty}\P\left(Q_k \in [\eps,1-\eps]\right) = 1-2\eps.\end{align*}
 and so 
 \begin{equation}\limsup_{k\to\infty} \min\{Q_k,1-Q_k\}>0 \text{  almost surely.}\label{eq:Qk}\end{equation}

\begin{figure}[t]
\centering
\includegraphics[width=0.6\textwidth]{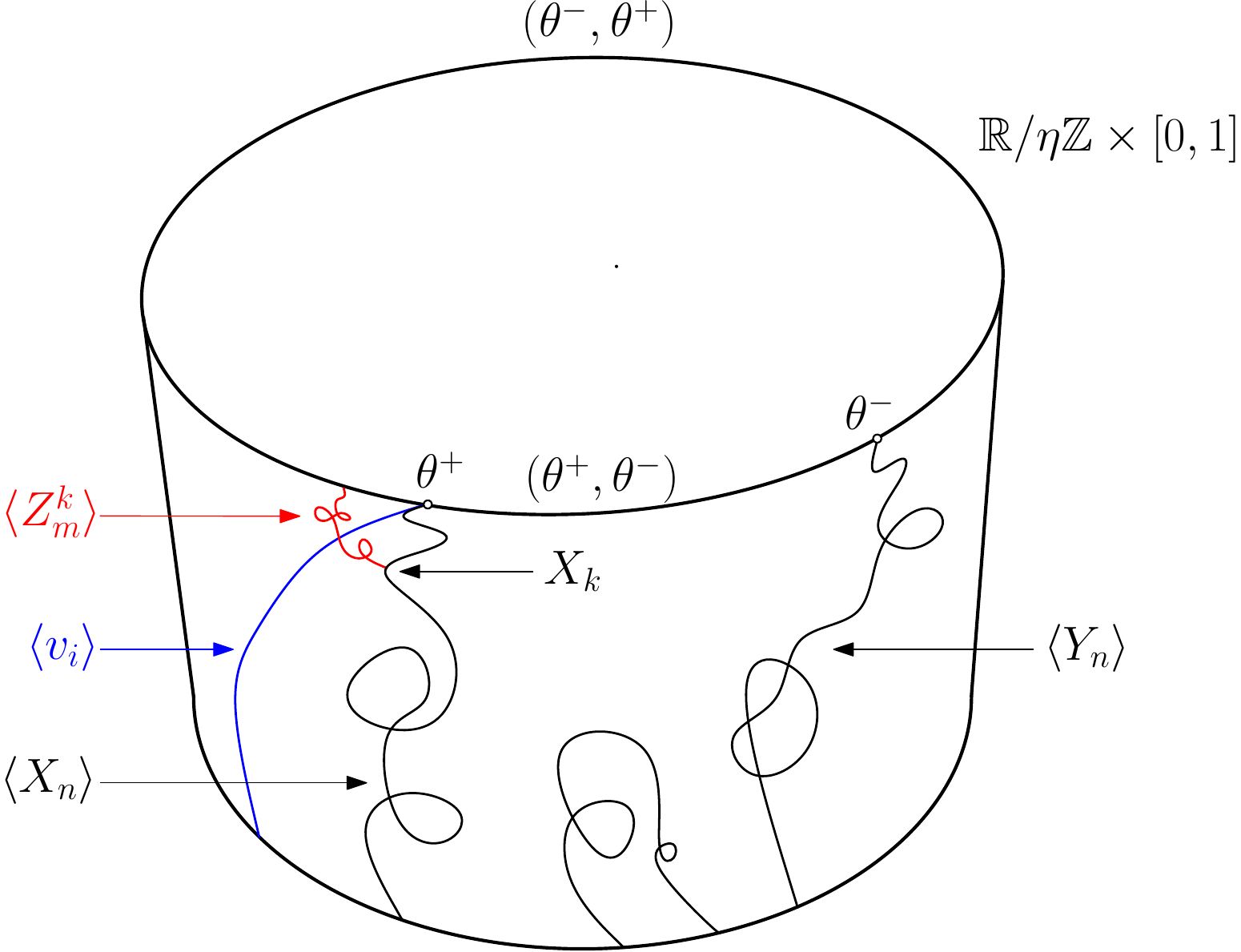}
\caption{Illustration of the proof. Conditioned on the random walk $\langle X_n \rangle_{n\geq0}$, there exists a random $\eps>0$ such that almost surely, for infinitely many $k$, a new random walk $\langle Z^k_m \rangle_{m\geq0}$ (red) started from $X_k$ has probability at least $\eps$ of hitting the path $\langle v_i \rangle_{i\geq0}$ (blue).}
\end{figure}

Let $\langle v_i \rangle_{i\geq0}$ be a path in $G$ started at $\rho$ such that $\lim_{i\to\infty}\theta(v_i)=\theta^+$ and the height $y(v_i)$ tends to $1$. 
Observe (Figure 1) that for every $k\geq 0$, the union of the traces $\{ v_i : i \geq 0\}\cup\{ Y_n : n\geq 0\}$ either contains $X_k$ or disconnects $X_k$ from at least one of the two intervals $(\theta^-,\theta^+)$ or $(\theta^+,\theta^-)$. That is, for at least one of the intervals $(\theta^-,\theta^+)$ or $(\theta^+,\theta^-)$, any path in $G$ started at $X_k$ that converges to this interval must intersect $\{ v_i : i \geq 0\}\cup\{ Y_n : n\geq 0\}$. It follows that
\begin{equation*}\P\left(\langle Z^k_m \rangle_{m\geq0} \text{ hits } \{v_i : i\geq 0\}\cup\{ Y_n :n\geq 0\} \,\middle|\, \langle X_n \rangle_{n\geq0},\, \langle Y_n\rangle_{n\geq0} \right) \geq \min\{Q_k,1-Q_k\}.\end{equation*}
and so, applying \eqref{eq:Qk},
\begin{equation}\label{eq:Poisson2}\limsup_{k\to\infty}\P\left(\langle Z^k_m \rangle_{m\geq0} \text{ hits } \{v_i : i\geq 0\}\cup\{ Y_n :n\geq 0\} \,\middle|\, \langle X_n \rangle_{n\geq0},\, \langle Y_n\rangle_{n\geq0} \right) >0\end{equation}
almost surely.
We next claim that
\begin{equation}\label{eq:Poisson1}\P\left(\langle Z^k_m \rangle_{m\geq0} \text{ hits } \{ Y_n :n\geq 0\} \,\middle|\, \langle X_n \rangle_{n\geq0},\, \langle Y_n\rangle_{n\geq0} \right) \xrightarrow[k\to\infty]{a.s.}0.\end{equation}
Indeed, we have that
\begin{multline*}\P\left(\langle Z^k_m \rangle_{m\geq0} \text{ hits } \{ Y_n :n\geq 0\} \,\middle|\, \langle X_n \rangle_{n\geq0},\, \langle Y_n\rangle_{n\geq0} \right)\\ = \P\left(\langle Z^k_m \rangle_{m\geq0} \text{ hits } \{ Y_n :n\geq 0\} \,\middle|\, \langle X_n \rangle_{n=0}^k,\, \langle Y_n\rangle_{n\geq0} \right)\\
= \P\left(\langle X_m \rangle_{m\geq k} \text{ hits } \{ Y_n :n\geq 0\} \,\middle|\, \langle X_n \rangle_{n=0}^k,\, \langle Y_n\rangle_{n\geq0} \right). \end{multline*}
The rightmost expression converges to zero almost surely by an easy application of L\'evy's 0-1 law:  For each $k_0 \leq k$, we have that
\begin{multline*}\P\left(\langle Z^k_m \rangle_{m\geq0} \text{ hits } \{ Y_n :n\geq 0\} \,\middle|\, \langle X_n \rangle_{n\geq0},\, \langle Y_n\rangle_{n\geq0} \right)\\ \leq \P\left(\langle X_m \rangle_{m\geq k_0} \text{ hits } \{ Y_n :n\geq 0\} \,\middle|\, \langle X_n \rangle_{n=0}^k,\, \langle Y_n\rangle_{n\geq0} \right) 
\\ \xrightarrow[k\to\infty]{a.s.} \mathbbm{1}(\langle X_m \rangle_{m\geq k_0} \text{ hits } \{ Y_n :n\geq 0\}), \end{multline*}
and the claim follows since $\mathbbm{1}(\langle X_m \rangle_{m\geq k_0} \text{ hits } \{ Y_n :n\geq 0\})\to0$ a.s.\ as $k_0\to\infty$.


Combining \eqref{eq:Poisson2} and \eqref{eq:Poisson1}, we deduce that
\begin{multline*}\limsup_{k \to \infty}\P\left(\langle Z^k_m \rangle_{m\geq0} \text{ hits } \{v_i : i\geq 0\} \,\middle|\, \langle X_n \rangle_{n\geq0} \right)\\
=\limsup_{k \to \infty}\P\left(\langle Z^k_m \rangle_{m\geq0} \text{ hits } \{v_i : i\geq 0\} \,\middle|\, \langle X_n \rangle_{n\geq0},\, \langle Y_n \rangle_{n\geq0} \right)>0\end{multline*}
 almost surely. Applying \cref{thm:criterion} with $\psi=(\theta,y):V\to \R/\eta\Z \times[0,1]$ completes the proof. 
\end{proof} 

\subsection{Proof of \cref{thm:roughembeddings}}
\begin{proof}
Let $G=(V,E)$ be a transient network with bounded local geometry such that there exists a proper plane graph $G'=(V',G')$ with bounded degrees and a rough isometry $\phi: V \to V'$. Then $G'$ is also transient \cite[\S2.6]{LP:book}. We may assume that $G'$ is simple, since $\phi$ remains a rough isometry if we modify $G'$ by deleting all self-loops and identifying all multiple edges between each pair of vertices. In this case, there exists a simple, bounded degree, proper plane triangulation $T'$ containing $G'$ as a subgraph \cite[Lemma 4.3]{BS96a}. The triangulation $T'$ is transient by Rayleigh monotonicty, and since it has bounded degrees it is CP hyperbolic by the He-Schramm Theorem. Let $\cC$ be a circle packing of $T'$ in the disc $\D$ and let $z$ be the embedding of $G'$ defined by this circle packing, so that for each $v\in V'$, $z(v')$ is the center of the circle of $\cC$ corresponding to $v'$, and each edge  of $G'$ is embedded as a straight line between the centers of the circles corresponding to its endpoints.

Recall that the \textbf{Dirichlet energy} of a function $f:V\to\R$ is defined by 
\[\cE_G(f)=\frac{1}{2}\sum_e c(e)\left|f(e^+)-f(e^-)\right|^2.\]

The following lemma is implicit in \cite{Soardi93}.

\begin{lem}\label{lem:roughenergy}
Let $\phi: V \to V'$ be a rough isometry from a network $G$ with bounded local geometry to a network $G'$ with bounded local geometry. Then there exists a constant $C$ such that
\begin{equation}\label{eq:roughenergy} \cE_G(f\circ\phi)\leq C\cE_{G'}(f).\end{equation}
for all functions $f:V'\to \R$.
\end{lem}

For each pair of adjacent vertices $u$ and $v$ in $G$, we have that $d'(\phi(u),\phi(v))\leq \alpha + \beta$, and so there exists a path in $G'$ from $\phi(u)$ to $\phi(v)$ of length at most $\alpha+\beta$. Fix one such path $\Phi(u,v)$ for each $u$ and $v$. 

\begin{proof} Since there at most $\alpha+\beta$ edges of $G$ in each path $\Phi(u,v)$, and the conductances of $G$ are bounded, there exists a constant $C_1$ such that
\begin{align}\cE_G(f\circ\phi) &= \sum_e c(e)\left(f\circ\phi(e^+)-f\circ\phi(e^-)\right)^2 = \sum_e c(e)\Big(\sum_{e'\in\Phi(e)}f(e'^+)-f(e'^-)\Big)^2 \nonumber\\
&\leq C_1 \sum_e \sum_{e'\in\Phi(e)}\left(f(e'^+)-f(e'^-)\right)^2.\end{align}
Let $e'$ be an edge of $G'$. If $e_1$ and $e_2$ are two edges of $G$ such that $\Phi(e_1^-,e_1^+)$ and $\Phi(e_1^-,e_1^+)$ both contain $e'$, then  $d'(\phi(e_1^-),\phi(e_2^-))\leq 2(\alpha+\beta)$ and so 
\begin{align*}d(e_1^-,e_2^-)\leq \alpha\left(d'\left(\phi(e_1^-),\phi(e_2^-)\right)+\beta\right) \leq \alpha(2\alpha+3\beta).\end{align*}
Thus, the set of edges $e$ of $G$ such that $\Phi(e^-,e^+)$ contains $e'$ is contained a ball of radius $\alpha(2\alpha+3\beta)$ in $G$.
Since $G$ has bounded degrees, the number of edges contained in such a ball is bounded by some constant $C_2$. Combining this with the assumption that the conductances of $G'$ are bounded below by some constant $C_3$, we have that
\begin{align*}
\cE_G(f\circ\phi) &\leq C_1 \sum_e \sum_{e'\in\Phi(e)}\left(f(e'^+)-f(e'^-)\right)^2\\ &\leq C_1 C_2 \sum_{e'}\left(f(e'^+)-f(e'^-)\right)^2 \leq C_1C_2C_3 \cE_{G'}(f).\qedhere
\end{align*}
\end{proof}

The proofs of Lemmas \ref{lem:convergence} and \ref{lem:noatoms}  below are adapted from \cite{ABGN14}. 
\begin{lem}\label{lem:convergence} Let $\langle X_n \rangle_{n\geq0}$ denote the random walk on $G$. Then $z\circ\phi(X_n)$ converges to a point in $\partial \D$ almost surely.
\end{lem}
\begin{proof} 
Ancona, Lyons and Peres \cite{AnLyPe99} proved for every function $f$ of finite Dirichlet energy on a transient network $G$, the sequence $f(X_n)$ converges almost surely as $n\to\infty$. Thus, it suffices to prove that each coordinate of $z\circ\phi$ has finite Dirichlet energy. Applying the inequality \eqref{eq:roughenergy}, it suffices to prove that each coordinate of $z$ has finite energy. For each vertex $v'$ of $G'$, let $r(v')$ denote the radius of the circle corresponding to $v'$ in $\cC$. Then, letting $z=(z_1,z_2)$,
\begin{align*} \cE_{G'}(z_1) + \cE_{G'}(z_2) &= \sum_{e'}\left|z(e^+)-z(e^-)\right|^2= C\sum_{e'}\left(r(e^+)+r(e^-)\right)^2
\\
&\leq 2\max\left(\deg(v')\right)C\sum_{v'}r(v)^2 \leq 2C\max\left(\deg(v')\right)<\infty \end{align*}
since $\sum \pi r(v)^2\leq\pi$ is the total area of all the circles in the packing.
\end{proof}

\begin{lem}\label{lem:noatoms}
The law of $\lim_{n\to\infty}z\circ\phi(X_n)$ does not have any atoms.
\end{lem}
\begin{proof} Let $B_k(\rho)$ be the set of vertices of $G$ at graph distance at most $k$ from $\rho$, and let $G_{k}$ be the subnetwork of $G$ induced by $B_k(\rho)$ (i.e. the induced subgraph together with the conductances inherited from $G$).  Recall that the \textbf{free effective conductance} between a set $A$ and a set $B$ in an infinite graph $G$ is given by
\[ \CeffF(A\leftrightarrow B \,;\, G) = \min \left\{\cE(F) : F(a)=1 \,\,\forall a \in A,\;\; F(b)=0 \,\,\forall b \in B\right\}.\]
The same variational formula also defines the effective conductance between two sets $A$ and $B$ in a finite network $G$,  denoted $\Ceff(A\leftrightarrow B;G)$. 
A related quantity is the effective conductance from a vertex $v$ to infinity in $G$
\[ \Ceff(v\to\infty; G):= \lim_{k\to\infty}\CeffF(v \leftrightarrow V \setminus B_k(\rho);G)=\lim_{k\to\infty}\CeffF(v \leftrightarrow V \setminus B_k(\rho);G_{k+1}),\] which is positive if and only if $G$ is transient. See \cite[\S2 and \S9]{LP:book} for background on electrical networks.  
The inequality \eqref{eq:roughenergy} above implies that there exists a constant $C$ such that
\begin{equation}\label{eq:conductance} \CeffF(A\leftrightarrow B \,;\, G) \leq C \CeffF\left(\phi(A)\leftrightarrow \phi(B)\,;\, G'\right)\end{equation}
for each two sets of vertices $A$ and $B$ in $G$.

Let $\rho \in V$ and $\xi \in \partial \D$ be fixed, and let \[A_\eps(\xi):=\{v \in V : |z\circ\phi(v) - \xi|\leq \eps\}.\] In \cite[Corollary 5.2]{ABGN14}, it is proven that 
\[\lim_{\eps\to0} \CeffF\left(\phi(\rho)\leftrightarrow \phi(A_\eps(\xi)) \,;\, T'\right) =0.\]
Applying \eqref{eq:conductance} and Rayleigh monotonicity, we have that
\begin{multline}\CeffF\left(\rho\leftrightarrow A_\eps(\xi) \,;\, G\right) \leq  C\CeffF\left(\phi(\rho)\leftrightarrow \phi(A_\eps(\xi)) \,;\, G'\right)\\ \leq  C\CeffF\left(\phi(\rho)\leftrightarrow \phi(A_\eps(\xi)) \,;\, T'\right) \xrightarrow[\eps\to0]{}0.\label{eq:convergence1}\end{multline}
The well-known inequality of \cite[Exercise 2.34]{LP:book} then implies that
\begin{align*}P_\rho(\text{hit } A_\eps(\xi) \text{ before } V\setminus B_k(\rho) ) \leq \frac{\Ceff\left(\rho\leftrightarrow A_\eps(\xi) \cap B_k(\rho);G_{k+1}\right)}{\Ceff\left(\rho\leftrightarrow A_\eps(\xi) \cup  V\setminus B_k(\rho);G_{k+1}\right)}\\ \leq \frac{\Ceff\left(\rho\leftrightarrow A_\eps(\xi) \cap B_k(\rho);G_{k+1}\right)}{\Ceff\left(\rho\leftrightarrow  V\setminus B_k(\rho);G_{k+1}\right)}. \end{align*}
Applying the exhaustion characterisation of the free effective conductance \cite[\S 9.1]{LP:book},
which states that
\[\CeffF(A\leftrightarrow B ; G ) = \lim_{k\to\infty}\Ceff(A\leftrightarrow B ; G_{k}),\] we have that
 \begin{align} 
P_\rho\left(\text{hit } A_\eps(\xi)\right) &=\lim_{k\to\infty}P_\rho(\text{hit } A_\eps(\xi) \text{ before } V\setminus B_k(\rho) )\nonumber\\
& \leq \lim_{k\to\infty} \frac{\Ceff\left(\rho\leftrightarrow A_\eps(\xi) \cap B_k(\rho);G_{k+1}\right)}{\Ceff\left(\rho\leftrightarrow V\setminus B_k(\rho);G_{k+1}\right)}
 =  \frac{\CeffF\left(\rho\leftrightarrow A_\eps(\xi);G\right)}{\Ceff\left(\rho\to \infty ;G\right)}. \label{eq:convergence2}
 \end{align}
Combining \eqref{eq:convergence1} and \eqref{eq:convergence2} we deduce that
\[P_\rho\left(\lim_{n\to\infty}z\circ\phi(X_n)=\xi\right) \leq  \lim_{\eps\to0} P_\rho\left(\text{hit } A_\eps(\xi)\right)=0. \qedhere\]
\end{proof}

This concludes the part of \cref{thm:roughembeddings} concerning the existence of an embedding. 

\begin{remark} 
The statement that  $\P(\text{hit }A_\eps(\xi))$ converges to zero as $\eps$ tends to zero  for every $\xi\in\partial\D$ can also be used to deduce convergence of $z\circ\phi(X_n)$ without appealing to the results of \cite{AnLyPe99}: Suppose for contradiction that with positive probability  $z\circ\phi(X_n)$ does not converge, so that there exist two boundary points $\xi_1,\xi_2\in\partial\D$ in the closure of  $\{z\circ\phi(X_n) : n\geq 0\}$. It is not hard to see that in this case the closure of $\{z\circ\phi(X_n) : n\geq 0\}$ must contain one of the boundary intervals $[\xi_1,\xi_2]$ or $[\xi_2,\xi_1]$. However, if the closure of $\{z\circ\phi(X_n) : n\geq 0\}$ contains an interval of positive length with positive probability, then there exists a point $\xi\in\partial\D$ that is contained in the closure of  $\{z\circ\phi(X_n) : n\geq 0\}$ with positive probability. This contradicts the convergence  of $\P(\text{hit }A_\eps(\xi))$ to zero.
\end{remark}

\noindent\textit{Proof of \cref{thm:roughembeddings}, identification of the Poisson boundary.} 
Let $\langle X_n \rangle_{n\geq0}$ and $\langle Y_n \rangle_{n\geq0}$ be independent random walks on $G$ started at $\rho$, and for each $k\geq 0$ let $\langle Z^k_{j} \rangle_{j\geq 0}$ be a random walk on $G$ started at $X_k$ that is conditionally independent of $\langle X_n \rangle_{n\geq0}$ and $\langle Y \rangle_{n\geq0}$ given $X_k$.   Let $\P$ denote the joint distribution of $\langle X_n \rangle_{n\geq0}$, $\langle Y_n \rangle_{n\geq0}$ and all of the random walks $\langle Z^k_m \rangle_{m\geq0}$.

Suppose that $z$ is an embedding of $G'$ in $\D$ such that $z\circ \phi (X_n)$ converges to a point in $\partial\D$ almost surely and the law of the limit is non-atomic. 
Since the law of $\lim_{n\to\infty}z\circ\phi(X_n)$ is non-atomic, the random variables $\xi^+=\lim_{n\to\infty}z \circ \phi (X_n)$ and $\xi^-=\lim_{n\to\infty}z \circ \phi(Y_n)$ are almost surely distinct. Let 
\[Q_k:=\P\left(\lim_{m\to\infty} z\circ\phi(Z^k_m) \in (\xi^-,\xi^+) \,\middle|\, \langle X_n \rangle_{n\geq0}, \langle Y_n \rangle_{n\geq0} \right).\] 
Using the non-atomicity of the law of $\lim_{n\to\infty}z\circ\phi(X_n)$, the same argument as in the proof of \cref{Thm:generalPoisson} also shows that
\begin{equation}\label{eq:roughpoisson0} \limsup_{k\to\infty}\min\{Q_k,1-Q_k\} >0 \text{ almost surely.} \end{equation}


\medskip

We now come to a part of the proof that requires more substantial modification.  Let
$\langle v_i \rangle_{i\geq0}$ be  a path in $G$ starting at $\rho$ such that $z\circ\phi(v_i) \to \xi^+$.

Given a path $\langle u_i \rangle_{i\geq0}$ in $G$, let $\Phi(\langle u_i \rangle_{i\geq0})$ denote the path in $G'$ formed by concatenating the paths $\Phi(u_i,u_{i+1})$ for $i\geq 0$. For each $k\geq0$, let $\tau_k$ be the first time $t$ such that the path $\Phi(Z^k_{t-1},Z^k_{t})$ intersects the union of the traces of the paths $\Phi(\langle v_i \rangle_{i\geq0})$ and $\Phi(\langle Y_n \rangle_{n\geq0})$.
Since the union of the traces of the paths $\Phi(\langle v_i \rangle_{i\geq0})$ and $\Phi(\langle Y_n \rangle_{n\geq0})$ either contains $\phi(X_k)$ or disconnects $\phi(X_k)$ from at least one of the two intervals $(\xi^-,\xi^+)$ or $(\xi^+,\xi^-)$, we have that
\begin{equation*}\P\left(\tau_k <\infty \,\middle|\, \langle X_n \rangle_{n\geq0}, \langle Y_n \rangle_{n\geq0} \right) \geq \min\{Q_k,1-Q_k\}\end{equation*}
and so
\begin{equation}\label{eq:roughpoisson1}\limsup_{k\to\infty}\P\left(\tau_k <\infty \,\middle|\, \langle X_n \rangle_{n\geq0}, \langle Y_n \rangle_{n\geq0} \right) >0\end{equation}
almost surely by \eqref{eq:roughpoisson0}. 

 On the event that $\tau_k$ is finite, by definition of $\Phi$, there exists a vertex $u\in \{v_i :i \geq0\}\cup\{Y_n : n \geq 0\}$ such that $d'(\phi(Z_\tau),\phi(u))\leq 2\alpha+2\beta$, and consequently
$d(Z_\tau,u)\leq \alpha(2\alpha+3\beta)$ 
since $\phi$ is a rough isometry. Since $G$ has bounded degrees and edge conductances bounded above and below, $c(e)/c(u)\geq \delta$ for some $\delta>0$. Thus, by the strong Markov property,
\begin{multline}\label{eq:roughpoisson2} \P\left(\langle Z^k_m \rangle_{m\geq0} \text{ hits } \{v_i : i\geq 0\}\cup\{Y_n : n\geq0\} \,\middle|\, \langle X_n \rangle_{n\geq0},\, \langle Y_n \rangle_{n\geq0},\, \{\tau_k<\infty\} \right)\\ \geq \delta^{\alpha(2\alpha+3\beta)} >0  \end{multline}
for all $k\geq 0$. Combining \eqref{eq:roughpoisson1} and \eqref{eq:roughpoisson2} yields
\begin{equation}\label{eq:roughpoisson3}\limsup_{k\to\infty}\P\left(\langle Z^k_m \rangle_{m\geq0} \text{ hits }\{v_i: i\geq 0\}\cup\{ Y_n :n\geq 0\} \,\middle|\,\langle X_n \rangle_{n\geq0},\langle Y_n \rangle_{n\geq0}\right) >0  \end{equation}
almost surely. 
The same argument as in the proof of \cref{Thm:generalPoisson} also implies that
\begin{equation}\label{eq:roughpoisson4}\P\left(\langle Z^k_m\rangle_{m\geq0}\text{ hits }\{ Y_n :n\geq 0\}\,\middle|\, \langle X_n \rangle_{n\geq0}, \langle Y_n \rangle_{n\geq0}\right) \xrightarrow[k\to\infty]{a.s.}0.\end{equation}
We conclude by combining \eqref{eq:roughpoisson3} with \eqref{eq:roughpoisson4} and applying \cref{thm:criterion} to $\psi=z\circ\phi:V\to\D\cup\partial\D$. \qedhere

\end{proof}

\section{Identification of the Martin Boundary}\label{Sec:Martin}

In this section, we prove \cref{thm:comparison} and deduce \cref{Thm:Martin}. We begin by proving that the rectangle tiling of a bounded degree triangulation with edge conductances bounded above and below does not have any accumulations of rectangles other than at the boundary circle $\R/\eta\Z\times\{1\}$.
\begin{prop}
Let $T$ be a transient, simple, proper plane triangulation with bounded local geometry. Then for every vertex $v$ of $T$ and every $\eps>0$, there exist at most finitely many vertices $u$ of $T$ such that the probability that a random walk started at $u$ visits $v$ is greater than $\eps$.
\end{prop}
\begin{proof}
Let $\cC$ be a circle packing of $T$ in the unit disc $\D$, with associated embedding $z$ of $T$. Benjamini and Schramm \cite[Lemma 5.3]{BS96a} proved that for every $\eps>0$ and $\kappa>0$, there exists $\delta>0$ such that for any $v\in \tilde V$ with $|z(v)|\geq 1-\delta$, the probability that a random walk from $v$ ever visits a vertex $u$ such that $|z(u)|\leq 1-\kappa$ is at most~$\eps$. (Their proof is given for $c \equiv 1$ but extends immediately to this setting.)
 By setting $\kappa=1-|z(v)|$, it follows that for every $\eps>0$, there exists $\delta>0$ such that for every vertex $u$ with $|z(u)|\geq 1-\delta$, the probability that a random walk started at $u$ hits $v$ is at most $\eps$. The claim follows since $|z(u)|\geq 1-\delta$ for all but finitely many vertices $u$ of $T$. \end{proof}

\begin{corollary}\label{cor:stslocallyfinite}
Let $T$ be a transient proper plane triangulation with bounded local geometry, and let $\cS_\rho$ be a rectangle tiling of $T$. Then for every  $t\in[0,1)$, the cylinder $\R/\eta\Z\times[0,t]$ intersects only finitely many rectangles $S(e)\in\cS_\rho$.
\end{corollary}

We also require the following simple geometric lemma. 

\begin{lemma}
\label{Lem:interpolation} Let $T$ be a transient proper plane triangulation with bounded local geometry, and let $\cS_\rho$ be the square tiling of $T$ in the cylinder $\R/\eta\Z\times[0,1]$. Then for every sequence of vertices  $\langle v_n \rangle_{n\geq0}$ such that $y(v_n)\to 1$ and $\theta(v_n)\to\theta_0$ for some $(\theta_0,y_0)$ as $n\to\infty$, there exists a path $\langle \gamma_n \rangle_{n\geq0}$ containing $\{v_n : n \geq 0\}$ such that $y(\gamma_n)\to 1$ and $\theta(\gamma_n)\to \theta_0$ as $n\to\infty$.
\end{lemma}


\begin{proof}
Define a sequence of vertices $\langle v'_n \rangle_{n\geq0}$ as follows.
If the interval $I(v_n)$ is non-degenerate (i.e. has positive length), let $v'_n=v_n$. Otherwise, 
let $\langle \eta_j \rangle_{j\geq0}$ be a path from $v_n$ to $\rho$ in $T$, let $j(n)$ be the smallest $j>0$ such that the interval $I(\eta_{j(n)})$ is non-degenerate, and set $v'_n = \eta_{j(n)}$. 
Observe that $y(v'_n)=y(v_n)$, that the interval $I(v'_n)$ contains the singleton $I(v_n)$, and that the path $\tilde \eta^n = \langle\eta^n_j\rangle_{j=0}^{j(n)}$ from $v_n$ to $v'_n$ in $T$ satisfies $(I(\tilde \eta_j),y(\tilde \eta_j))=(I(v_n),y(v_n))$ for all $j<j(n)$.

Let $\theta'(v'_n) \in I(v_n)$ be chosen so that the line segment $\ell_n$ between the points $(\theta'(v'_n),y(v'_n))$ and $(\theta'(v'_{n+1}),y(v'_{n+1}))$ 
in the cylinder $\R/\eta\Z\times[0,1]$ does not intersect any degenerate rectangles of $\cS_\rho$ or a corner of any rectangles of $\cS_\rho$ (this is a.s. the case if $\theta'(v'_n)$ is chosen uniformly from $I(v_n)$ for each $n$). 
The line segment $\ell_n$ intersects some finite sequence of squares corresponding to non-degenerate rectangles of $\cS_\rho$, which we denote $e^n_1,\ldots,e^n_{l(n)}$. Let $e^n_i$ be oriented so that $y(e^{n+}_i)>y(e^{n-}_i)$ for every $i$ and $n$. For each $1 \leq i \leq l(n)-1$, either
\begin{enumerate}
\item 
 the vertical sides of the rectangles $S(e^n_i)$ and $S(e^n_{i+1})$ have non-disjoint interiors, in which case $e^n_i$ and $e^n_{i+1}$ lie in the boundary of a common face of $T$, or
 \item the horizontal sides of the rectangles $S(e^n_i)$ and $S(e^n_{i+1})$ have non-disjoint interiors, in which case $e^n_i$ and $e^n_{i+1}$ share a common endpoint.
\end{enumerate}
In either case, since $T$ is a triangulation, there exists an edge in $T$ connecting $e^{n+}_i$ to $e^{n-}_{i+1}$ for each $i$ and $n$. 
We define a path $\gamma^n$ by alternatingly concatenating  the edges $e_i^n$ and the edges connecting $e^{n+}_i$ to $e^{n-}_{i+1}$ as $i$ increases from 1 to $l(n)$. Define the path $\gamma$ by  concatenating all of the paths
 \[ \tilde\eta^n\circ\gamma^n\circ(-\tilde\eta^{n+1}) \]
 where $-\tilde\eta^{n+1}$ denotes the reversal of the path $\tilde\eta^{n+1}$. 

 Let $M$ be an upper bound for the degrees of $T$ and for the conductances and resistances of the edges of $T$.
For each vertex $v$ of $T$, there are at most $M$ rectangles of $\cS_\rho$ adjacent to $I(v)$ from above, so that at least one of these rectangles has width at least $\mathrm{length}(I(v))/M$. This rectangle must have height at least $\mathrm{length}(I(v))/M^2$. It follows that 
\begin{equation}\label{eq:path1}\mathrm{length}(I(v)) \leq M^{-2}\left(1-y(v)\right)\end{equation}
for all $v\in T$, and so for every edge $e$ of $T$,
\begin{equation}\label{eq:path2}|y(e^-)-y(e^+)| \leq \frac{1}{c(e)}\mathrm{length}\left(I(e^-)\right)\leq M^{-3}\left(1-y(e^-)\right)\end{equation}

By construction, every vertex $w$ visited by the path $\tilde\eta^n\circ\gamma^n\circ(-\tilde\eta^{n+1})$ has an edge emanating from it such that the associated rectangle intersects the line segment $\ell_n$, and consequently a neighbouring vertex $w'$ such that $y(w') \geq \min\{y(v_n),y(v_{n+1})\}$. Applying \eqref{eq:path2}, we deduce that
\[ y(w) \geq (1+M^{-3})\min\{y(v_n),y(v_{n+1})\}-M^{-3}\]
for all vertices $w$ visited by the path $\tilde\eta^n\circ\gamma^n\circ(-\tilde\eta^{n+1})$, and consequently $y(\gamma_k)\to1$ as $k\to\infty$. The estimate \eqref{eq:path2} then implies that $\mathrm{length}(I(\gamma_k))\to0$ as $k\to\infty$. Since $I(w)$ intersects the projection to the boundary circle of the line segment $\ell_n$ for each vertex $w$ visited by the path  $\tilde\eta^n\circ\gamma^n\circ(-\tilde\eta^{n+1})$, we deduce that $\theta(\gamma_k)\to\theta_0$ as $k\to\infty$. \qedhere
\end{proof}

We also have the following similar lemma for circle packings.

\begin{lemma}\label{lem:interpolation2} Let $T$ be a CP hyperbolic plane triangulation, and 
let $\cC$ be a circle packing of $T$ in $\D$ with associated embedding $z$. 
Then for every sequence $\langle v_n \rangle_{n\geq0}$ such that $z(v_n)$ converges as $n\to\infty$, there exists a path $\langle \gamma_n \rangle_{n\geq0}$ in $T$ containing $\{v_n : n \geq 0\}$ such that $z(\gamma_n)$ converges as $n\to\infty$.
\end{lemma}

\begin{proof}
The proof is similar to that of \cref{Lem:interpolation}, and we provide only a sketch. Draw a straight line segment in $\D$ between the centres of the circles corresponding to each consecutive pair of vertices $v_n$ and $v_{n+1}$. The set of circles intersected by the line segment contains a path $\gamma^n$ in $T$ from $v_n$ to $v_{n+1}$. (If this line segment is not tangent to any of the circles of $\cC$, then the set of circles intersected by the segment is exactly a path in $T$.)  The path $\gamma$ is defined by concatenating the paths $\gamma^n$. For every $\eps>0$, there are at most finitely many $v$ for which the radius of the circle corresponding to $v$ is greater than $\eps$, since the sum of the squared radii of all the circles in the packing is at most 1. Thus, for large $n$, all the circles corresponding to vertices used by the path $\gamma^n$ are small.  The circles that are intersected by the line segment between $z(v_n)$ and $z(v_{n+1})$ therefore necessarily have centers close to $\xi_0$ for large $n$. We deduce that $z(\gamma_i)\to \xi_0$ as $i \to \infty$. \qedhere


\end{proof}
\begin{proof}[Proof of \cref{thm:comparison}.]
Let $\langle X_n \rangle_{n\geq0}$ be a random walk on $T$. Our assumptions guarantee that $\theta(X_n)$ and $z(X_n)$ both converge almost surely as $n\to\infty$ and that the laws of these limits are both non-atomic and have support $\R/\eta\Z$ and $\partial\D$ respectively.

Suppose that $\gamma$ is a path in $T$ that visits each vertex at most finitely often. We claim that  $\theta(\gamma_i)$ converges if and only if $z(\gamma_i)$ converges if and only if $X_n$ almost surely does not hit $\gamma_i$ infinitely often. We prove that $\theta(\gamma_i)$ converges if and only if $X_n$ almost surely does not hit $\gamma_i$ infinitely often. The proof for $z(\gamma_i)$ is similar.

If $\theta(\gamma_i)$ converges, then $X_n$ almost surely does not hit $\gamma_i$ infinitely often, since otherwise $\lim_{i\to\infty} \theta(\gamma_i)$ would be an atom in the law of $\lim_{n\to\infty} \theta(X_n)$.
Conversely, if $\theta(\gamma_i)$ does not converge, then  there exist at least two distinct points $\theta_1,\theta_2\in\R/\eta\Z$ such that $\theta_1\times\{1\}$ and $\theta_2\times\{1\}$ are in the closure of $\{(\theta(\gamma_i),y(\gamma_i)):n\geq0\}$. Let $\langle Y_n \rangle_{n\geq 0}$ be a random walks started at $\rho$ independent of $\langle X_n \rangle_{n\geq0}$. Since the law of $\lim_{n\to\infty}\theta(X_n)$ has full support, we have with positive probability that $\lim_{n\to\infty}\theta(X_n)\in(\theta_1,\theta_2)$ and $\lim_{n\to\infty}\theta(Y_n)\in(\theta_2,\theta_1)$. 
On this event, the union of the traces $\{ X_n : n\geq0\}\cup\{Y_n : n \geq 0\}$ disconnects $\theta_1\times\{1\}$ from $\theta_2\times\{1\}$, and consequently the path $\langle \gamma_i \rangle_{n\geq0}$ must hit  $\{ X_n : n\geq0\}\cup\{Y_n : n \geq 0\}$ infinitely often.
By symmetry, there is a positive probability that $\langle X_n \rangle_{n\geq0}$ hits the trace $\{\gamma_i :n\geq0\}$ infinitely often as claimed. 
\medskip

We deduce that for every path $\gamma$ in $T$ that visits each vertex of $T$ at most finitely often, $\theta(\gamma_i)$ converges if and only if $z(\gamma_i)$ converges. It follows from this and \cref{Lem:interpolation,lem:interpolation2} that for any sequence of vertices $\langle v_i \rangle_{i\geq0}$ in $T$ that includes each vertex of $T$ at most finitely often, the sequence $\theta(v_i)$ converges if and only if $z(v_i)$ converges, and hence that 
the map $\xi \mapsto \theta(\xi)=\lim_{z(v)\to\xi} \theta(v)$ 
is well defined.
To see that this map is a homeomorphism, suppose that $\xi_n$ is a sequence of points in $\partial \D$ converging to $\xi$. For every $n$, there exists a vertex $v_n\in V$ such that $|\xi_n-z(v_n)|\leq 1/n$ and $|\theta(\xi_n)-\theta(v_n)|\leq 1/n$. Thus, $z(v_n)\to\xi$ and we have
\[|\theta(\xi)-\theta(\xi_n)|\leq |\theta(\xi)-\theta(v_n)|+|\theta(v_n)-\theta(\xi_n)|\xrightarrow[n\to\infty]{}0.\]
The proof of the continuity of the inverse is similar.
\end{proof}




\begin{proof}[Proof of \cref{Thm:Martin}]
By \cite[Theorem 1.2]{ABGN14}, a sequence of vertices $\langle v_i \rangle_{i\geq0}$ converges to a point in the Martin boundary of $T$ if and only if $z(v_i)$ converges to a point in $\partial\D$, and the map 
\[ \xi \mapsto M_\xi := \lim_{\i\to\infty} M_{v_i} \text{ where } z(v_i) \to \xi  \]
is a homeomorphism from $\partial \D$ to $\partial \cM (T)$. Combining this with \cref{thm:comparison} completes the proof. 
\end{proof}

\subsection*{Acknowledgements}
This work was carried out while TH was an intern at Microsoft Research, Redmond. We thank Russ Lyons and Asaf Nachmias for helpful discussions. We thank Itai Benjamini for granting us permission to include the square tiling of Figure 1, which originally appeared in \cite{BS96b}. The circle packing in Figure 1 was created using Ken Stephenson's CirclePack software. 

\bibliographystyle{abbrv}
\bibliography{unimodular}

\begin{thebibliography}{10}

\bibitem{AnLyPe99}
A.~Ancona, R.~Lyons, and Y.~Peres.
\newblock Crossing estimates and convergence of {D}irichlet functions along
  random walk and diffusion paths.
\newblock {\em Ann. Probab.}, 27(2):970--989, 1999.

\bibitem{ABGN14}
O.~Angel, M.~T. Barlow, O.~Gurel-Gurevich, A.~Nachmias, et~al.
\newblock Boundaries of planar graphs, via circle packings.
\newblock {\em The Annals of Probability}, 44(3):1956--1984, 2016.

\bibitem{AHNR15}
O.~Angel, T.~Hutchcroft, A.~Nachmias, and G.~Ray.
\newblock Unimodular hyperbolic triangulations: Circle packing and random walk.
\newblock {\em Inventiones mathematicae}, pages 1--40.

\bibitem{BS96a}
I.~Benjamini and O.~Schramm.
\newblock Harmonic functions on planar and almost planar graphs and manifolds,
  via circle packings.
\newblock {\em Invent. Math.}, 126(3):565--587, 1996.

\bibitem{BS96b}
I.~Benjamini and O.~Schramm.
\newblock Random walks and harmonic functions on infinite planar graphs using
  square tilings.
\newblock {\em Ann. Probab.}, 24(3):1219--1238, 1996.

\bibitem{blackwell1955transient}
D.~Blackwell.
\newblock On transient markov processes with a countable number of states and
  stationary transition probabilities.
\newblock {\em The Annals of Mathematical Statistics}, pages 654--658, 1955.

\bibitem{BSST40}
R.~L. Brooks, C.~A.~B. Smith, A.~H. Stone, and W.~T. Tutte.
\newblock The dissection of rectangles into squares.
\newblock {\em Duke Math. J.}, 7:312--340, 1940.

\bibitem{Dynkin69}
E.~B. Dynkin.
\newblock The boundary theory of {M}arkov processes (discrete case).
\newblock {\em Uspehi Mat. Nauk}, 24(2 (146)):3--42, 1969.

\bibitem{G13}
A.~Georgakopoulos.
\newblock The boundary of a square tiling of a graph coincides with the poisson
  boundary.
\newblock {\em Inventiones mathematicae}, 203(3):773--821, 2016.

\bibitem{Gromov81}
M.~Gromov.
\newblock Hyperbolic manifolds, groups and actions.
\newblock In {\em Riemann surfaces and related topics: {P}roceedings of the
  1978 {S}tony {B}rook {C}onference ({S}tate {U}niv. {N}ew {Y}ork, {S}tony
  {B}rook, {N}.{Y}., 1978)}, volume~97 of {\em Ann. of Math. Stud.}, pages
  183--213. Princeton Univ. Press, Princeton, N.J., 1981.

\bibitem{Gromov87}
M.~Gromov.
\newblock Hyperbolic groups.
\newblock In {\em Essays in group theory}, volume~8 of {\em Math. Sci. Res.
  Inst. Publ.}, pages 75--263. Springer, New York, 1987.

\bibitem{HeSc}
Z.-X. He and O.~Schramm.
\newblock Hyperbolic and parabolic packings.
\newblock {\em Discrete Comput. Geom.}, 14(2):123--149, 1995.

\bibitem{Kanai85}
M.~Kanai.
\newblock Rough isometries, and combinatorial approximations of geometries of
  noncompact {R}iemannian manifolds.
\newblock {\em J. Math. Soc. Japan}, 37(3):391--413, 1985.

\bibitem{K36}
P.~Koebe.
\newblock {\em Kontaktprobleme der konformen Abbildung}.
\newblock Hirzel, 1936.

\bibitem{LP:book}
R.~Lyons and Y.~Peres.
\newblock {\em Probability on Trees and Networks}.
\newblock Cambridge University Press, 2016.
\newblock Available at \url{http://pages.iu.edu/~rdlyons/}.

\bibitem{RodgersWilliamsbookV1}
L.~C.~G. Rogers and D.~Williams.
\newblock {\em Diffusions, {M}arkov processes, and martingales. {V}ol. 1}.
\newblock Cambridge Mathematical Library. Cambridge University Press,
  Cambridge, 2000.
\newblock Foundations, Reprint of the second (1994) edition.

\bibitem{Rohde11}
S.~Rohde.
\newblock {Oded Schramm: From Circle Packing to SLE}.
\newblock {\em Ann. Probab.}, 39:1621--1667, 2011.

\bibitem{Soardi93}
P.~M. Soardi.
\newblock Rough isometries and {D}irichlet finite harmonic functions on graphs.
\newblock {\em Proc. Amer. Math. Soc.}, 119(4):1239--1248, 1993.

\bibitem{Soardibook}
P.~M. Soardi.
\newblock {\em Potential theory on infinite networks}, volume 1590 of {\em
  Lecture Notes in Mathematics}.
\newblock Springer-Verlag, Berlin, 1994.

\bibitem{St05}
K.~Stephenson.
\newblock {\em Introduction to circle packing}.
\newblock Cambridge University Press, Cambridge, 2005.
\newblock The theory of discrete analytic functions.

\bibitem{Th78}
W.~P. Thurston.
\newblock The geometry and topology of 3-manifolds.
\newblock {\em Princeton lecture notes.}, 1978-1981.

\bibitem{Woess}
W.~Woess.
\newblock {\em Random walks on infinite graphs and groups}, volume 138 of {\em
  Cambridge Tracts in Mathematics}.
\newblock Cambridge University Press, Cambridge, 2000.

\end{thebibliography}

\end{document}